\documentclass[10pt, twoside, reqno]{amsart}
\usepackage{xcolor, amstext, amsfonts, amssymb, amsbsy, latexsym}
\usepackage{enumerate}
\usepackage[T1]{fontenc}
\usepackage{xy, hhline}
\newtheorem{lemma}{Lemma}[section]
\newtheorem{theo}[lemma]{Theorem}
\newtheorem{prop}[lemma]{Proposition}

\newtheorem{cor}[lemma]{Corollary}
\newtheorem{conj}[lemma]{Conjecture}
\theoremstyle{definition}

\newtheorem{remark}[lemma]{Remark}

\numberwithin{equation}{section}

\newenvironment{proof_of}[1]{\medskip\noindent{\bf Proof #1}}
{$\Box$ \bigskip}
\newenvironment{eq}{\begin{equation}}{\end{equation}}
\renewcommand{\Ref}[1]{(\ref{#1})}

\newcommand{\Char}{\mathop{\rm char}}

\newcommand{\trdeg}{\mathop{\rm tr.deg}}

\newcommand{\FF}{\mathbb{F}}

\newcommand{\CC}{\mathbb{C}}
\newcommand{\ZZ}{\mathbb{Z}}
\newcommand{\QQ}{\mathbb{Q}}
\newcommand{\NN}{\mathbb{N}}
\newcommand{\MM}{\mathbb{M}}
\newcommand{\SE}{\mathbb{S}}
\newcommand{\XX}{\mathbb{X}}

\newcommand{\algA}{\mathcal{A}}
\newcommand{\algB}{\mathcal{B}}
\newcommand{\algX}[1]{\alg_{\FF}\{X\}_{#1}}
\newcommand{\algZ}[1]{\alg_{\FF}\{Z\}_{#1}}
\newcommand{\mult}{\circ}
\newcommand{\Sym}{{\mathcal S}}

\newcommand{\K}[1]{K_{#1}}  
\newcommand{\KO}[1]{K_{0,#1}} 

\newcommand{\set}[1]{S_{#1}}  
\newcommand{\setO}[1]{S_{0,#1}} 

\newcommand{\I}[1]{I_{#1}}  
\newcommand{\IO}[1]{I_{0,#1}}  
\newcommand{\Ker}{{\rm Ker}}

\newcommand{\rel}[1]{\xi_{#1}}  
\newcommand{\relO}[1]{\xi_{0,#1}}  

\newcommand{\si}{\sigma}
\newcommand{\al}{\alpha}
\newcommand{\be}{\beta}
\newcommand{\ga}{\gamma}
\newcommand{\la}{\lambda}
\newcommand{\de}{\delta}
\newcommand{\De}{\Delta}

\newcommand{\LA}{\langle}
\newcommand{\RA}{\rangle}

\newcommand{\ov}[1]{\overline{#1}}
\newcommand{\un}[1]{{\underline{#1}} }
\newcommand{\alg}{\mathop{\rm alg}}
\newcommand{\mdeg}{\mathop{\rm mdeg}}
\newcommand{\Stab}{\mathop{\rm Stab}}
\newcommand{\tr}{\mathop{\rm tr}}
\newcommand{\rk}{\mathop{\rm rk}}
\newcommand{\Aut}{\mathop{\rm Aut}}

\newcommand{\G}{{\rm G}_2}
\newcommand{\SL}{{\rm SL}}
\newcommand{\GL}{{\rm GL}}
\newcommand{\Ort}{{\rm O}}
\newcommand{\spaceV}{\mathcal{V}}    

\newcommand{\matr}[4]{\left(\begin{array}{cc}
#1 & #2 \\
#3 & #4 \\
\end{array}\right)}

\newcommand{\OO}{\mathbf{O}}
\newcommand{\uu}{\mathbf{u}}
\newcommand{\vv}{\mathbf{v}}
\newcommand{\cc}{\mathbf{c}}
\newcommand{\zero}{\mathbf{0}}

\newcommand{\mL}[1]{\marginpar{\addtolength{\baselineskip}{-2pt}{\footnotesize {\color{blue}  #1}}}}

\usepackage{color}
\newcommand{\tb}[1]{#1}  

\begin{document}
\title[Separating $\G$-invariants of several octonions]{Separating $\G$-invariants of several octonions}

\thanks{The work was supported by UAEU grant G00003324 and partially supported in accordance with the state task of the IM SB RAS, project FWNF-2022-0003}

\author{Artem Lopatin}
\address{Artem Lopatin\\
State University of Campinas, 651 Sergio Buarque de Holanda, 13083-859 Campinas, SP, Brazil}
\email{dr.artem.lopatin@gmail.com (Artem Lopatin)}

\author{Alexandr N. Zubkov}
\address{Alexandr N. Zubkov\\
Department of Mathematical Sciences,
United Arab Emirates University,
Al Ain, Abu Dhabi, United Arab Emirates; 
Sobolev Institute of Mathematics, Omsk Branch, Pevtzova 13, 644043 Omsk, Russia
}
\email{a.zubkov@yahoo.com (Alexandr N. Zubkov)}

\begin{abstract}
We describe separating $\G$-invariants of several copies of the algebra of octonions over an algebraically closed field of characteristic two. We also obtain a minimal separating and a minimal generating \tb{set} for  $\G$-invariants of several copies of the algebra of octonions in case of a field of odd characteristic.


\end{abstract}

\maketitle

\section{Introduction}

All vector spaces and algebras are considered over an algebraically closed field $\FF$ of arbitrary characteristic $p=\Char\FF\geq0$. 

We continue the study of the invariants of the diagonal action of the exceptional simple group $\G$ on the space of several octonions, over a field of positive characteristic. Over the field of complex numbers, this was done in \cite{schwarz1988}. This result has been generalized to an arbitrary infinite field of odd characteristic in \cite{zubkov2018}, using a much finer technique of modules with good filtration, together with some results from the theory of groups with triality.

Unfortunately, the technique of modules with good filtration no longer works over a field of characteristic two and the complete description of the generating invariants in this case seems to be \tb{an} an extremely difficult problem. Thus, it makes sense to describe separating invariants, since they satisfy the most important property of ordinary invariants to separate closed orbits in the Zariski topology.  The latter problem is usually more accessible and it does not require extremely technical methods. We describe the separating invariants over an algebraically closed field of characteristic two, using a detailed description of the subalgebras of the octonion algebra (up to the action of $\G$) and the Hilbert-Mumford criterion (the part ``if''{}, see Section~\ref{subsection_1gr}).

The article is organized as follows. In Sections~\ref{section_O} and \ref{section_G2} we define the octonion algebra $\OO$, the group $\G$ and the algebra of $\G$-invariants $\FF[\OO^n]^{\G}$ of $n$ copies of the algebra of octonions $\OO$. We use notations from~\cite{zubkov2018}. Generators and relations between generators for $\FF[\OO^n]^{\G}$ were described by Schwarz~\cite{schwarz1988} over $\FF=\CC$. Zubkov and Shestakov described generators for $\FF[\OO^n]^{\G}$ over an arbitrary field with $\Char\FF\neq2$ (see Section~\ref{section_known}), but generators for the algebra  $\FF[\OO^n]^{\G}$ are still not known in case $p=2$.  Note that the invariants for the action of $F_4$ on several copies of the split Albert algebra were studied in~\cite{shestakov1996}. Our results are formulated in Section~\ref{section_new}. In Section~\ref{section_def} some definitions and notations are given. In Section~\ref{section_MGS} we describe a minimal generating and a minimal separating set for $\FF[\OO^n]^{\G}$ in case $p\neq2$. In Section~\ref{section_trace} a minimal generating set is constructed for the subalgebra $T_n\subset \FF[\OO^n]^{\G}$ of trace invariants in case $p=2$. In Section~\ref{section_subalgebras} subalgebras of $\OO$ of dimension $\leq3$ are described modulo $\G$-action in case $p=2$. This result is applied in Section~\ref{section_separ2} to obtain our main result which is the description of a separating set for $\FF[\OO^n]^{\G}$ in case $p=2$.

\section{Invariants of Octonions}\label{section_inv_O}

\subsection{Octonions}\label{section_O}

The {\it octonion algebra} $\OO=\OO(\FF)$, also known as the {\it split Cayley algebra}, is the vector space of all matrices

$$a=\matr{\al}{\uu}{\vv}{\be}\text{ with }\al,\be\in\FF \text{ and } \uu,\vv\in\FF^3,$$%
endowed with the following multiplication:
$$a a'  =
\matr{\al\al'+ \uu\cdot \vv'}{\al \uu' + \be'\uu - \vv\times \vv'}{\al'\vv +\be\vv' + \uu\times \uu'}{\be\be' + \vv\cdot\uu'},\text{ where } a'=\matr{\al'}{\uu'}{\vv'}{\be'},$$%
$\uu\cdot \vv = u_1v_1 + u_2v_2 + u_3v_3$ and $\uu\times \vv = (u_2v_3-u_3v_2, u_3v_1-u_1v_3, u_1v_2 - u_2v_1)$. For short, denote  $\cc_1=(1,0,0)$, $\cc_2=(0,1,0)$,  $\cc_3=(0,0,1)$, $\zero=(0,0,0)$ from $\FF^3$. Consider the following basis of $\OO$: $$e_1=\matr{1}{\zero}{\zero}{0},\; e_2=\matr{0}{\zero}{\zero}{1},\; \uu_i=\matr{0}{\cc_i}{\zero}{0},\;\vv_i=\matr{0}{\zero}{\cc_i}{0}$$
for $i=1,2,3$. The unity of $\OO$ is denoted by $1_{\OO}=e_1+e_2$. We identify octonions
$$\al 1_{\OO},\;\matr{0}{\uu}{\zero}{0},\; \matr{0}{\zero}{\vv}{0}$$
with $\al\in\FF$, $\uu,\vv\in\FF^3$, respectively. Similarly to $\OO(\FF)$ we define the algebra of octonions $\OO(\algA)$ over any commutative associative $\FF$-algebra $\algA$.

The algebra $\OO$ has a linear involution
$$\ov{a}=\matr{\be}{-\uu}{-\vv}{\al},\text{ satisfying  }\ov{aa'}=\ov{a'}\ov{a},$$%
a {\it norm} $n(a)=a\ov{a}=\al\be-\uu\cdot \vv$, and a non-degenerate symmetric bilinear {\it form} $q(a,a')=n(a+a')-n(a)-n(a')=\al\be' + \al'\be - \uu\cdot \vv' - \uu'\cdot \vv$. Define the linear function {\it trace} by $\tr(a)=a + \ov{a} = \al+\be$. The subspace $\{a\in\OO\,|\,\tr(a)=0\}$ of traceless octonions is denoted by $\OO_0$. Notice that
\begin{eq}\label{eq1}
\tr(aa')=\tr(a'a) \text{ and } n(aa')=n(a)n(a').
\end{eq}%
The following quadratic equation holds:
\begin{eq}\label{eq2}
a^2 - \tr(a) a + n(a) = 0.
\end{eq}%
Since 
\begin{eq}\label{eq3a}
n(a+a')=n(a)+n(a')-\tr(aa') + \tr(a)\tr(a'),
\end{eq}%
the linearization of equation~\Ref{eq2} implies
\begin{eq}\label{eq3}
aa' + a'a - \tr(a) a' - \tr(a') a  -\tr(aa') + \tr(a)\tr(a') = 0.
\end{eq}%
The algebra $\OO$ is a simple {\it alternative} algebra, i.e., the following identities hold for $a,b\in\OO$:
\begin{eq}\label{eq4}
a(ab)=(aa)b,\;\; (ba)a=b(aa).
\end{eq}%
The linearization implies that
\begin{eq}\label{eq5}
a(a'b) + a'(ab)=(aa'+a'a)b,\;\; (ba)a' + (ba')a = b(aa'+a'a).
\end{eq}%
The trace is associative, i.e., for all $a,b,c\in\OO$ we have
\begin{eq}\label{eq6}
\tr((ab)c) = \tr(a(bc)).
\end{eq}%
Note that
\begin{eq}\label{eq_n}
2n(a)=-\tr(a^2)+{\tr}^2(a) \text{ for each }a\in\OO.
\end{eq}%
More details on $\OO$ can be found in Sections 1 and 3 of~\cite{zubkov2018}.

\subsection{The group $\G$}\label{section_G2} The group $\G=\G(\FF)$ is known to be the group $\Aut(\OO)$ of all automorphisms of the algebra $\OO$. The group $\G$ contains a Zariski closed subgroup $\SL_3=\SL_3(\FF)$. Namely, every $g\in\SL_3$ defines the following automorphism of $\OO$:
$$a\to \matr{\al}{\uu g}{\vv g^{-T}}{\be},$$
where $g^{-T}$ stands for $(g^{-1})^T$ and $\uu,\vv\in\FF^3$ are considered as row vectors. In what follows $\SL_3$ is regarded as this subgroup of $\G$. For every $\uu,\vv\in \OO$ define $\de_1(\uu),\de_2(\vv)$ from $\Aut(\OO)$ as follows:
$$\de_1(\uu)(a')=\matr{\al' - \uu\cdot \vv'}{(\al'-\be' - \uu\cdot \vv')\uu + \uu'}{\vv' - \uu'\times \uu}{\be' + \uu\cdot\vv'},$$
$$\de_2(\vv)(a')=\matr{\al' + \uu'\cdot \vv}{\uu' + \vv'\times \vv}{(-\al'+\be' - \uu'\cdot \vv)\vv + \vv'}{\be' - \uu'\cdot\vv}.$$
The group $\G$ is generated by  $\SL_3$ and $\de_1(t\uu_i),\de_2(t\vv_i)$ for all $t\in\FF$ and $i=1,2,3$ \tb{(for example, see Section 3 of~\cite{zubkov2018})}. By straightforward calculations we can see that
\begin{eq}\label{eq_h}
\hbar:\OO\to\OO, \text{ defined by }a \to \matr{\be}{-\vv}{-\uu}{\al},
\end{eq}%
belongs to $\G$ (see also the proof of Lemma 1 of~\cite{zubkov2018}).

The action of $\G$ on $\OO$ satisfies the following properties:
$$\ov{ga}=g\ov{a},\; \tr(ga)=\tr(a),\;n(ga)=n(a),\; q(ga,ga')=q(a,a').$$%
Thus, $\OO_0$ is a $\G$-submodule of $\OO$.

Consider the diagonal action of $\G$ on the vector space $\OO^n=\OO\oplus \cdots\oplus\OO$ ($n$ copies), i.e., $g(a_1,\ldots,a_n)=(ga_1,\ldots,g a_n)$ for all $g\in \G$ and $a_1,\ldots,a_n\in\OO$. The coordinate ring of the affine variety $\OO^n$ is the polynomial $\FF$-algebra $\K{n}=\FF[\OO^n]=\FF[z_{ij}\,|\,1\leq i\leq n,\; 1\leq j\leq 8]$, where $z_{ij}:\OO^n\to\FF$ is defined by $(a_1,\ldots,a_n)\to \al_{ij}$ for 
\begin{eq}\label{eq7}
a_i=\matr{\al_{i1}}{(\al_{i2}, \al_{i3}, \al_{i4})}{(\al_{i5}, \al_{i6}, \al_{i7})}{\al_{i8}}\in\OO.
\end{eq}%
The action of $\GL(\OO)$ on $\OO$ induces the action on $\K{n}$ by $(gf)(\un{a})=f(g^{-1}\un{a})$ for all $g\in\GL(\OO)$, $f\in\K{n}$, $\un{a}\in\OO^n$.

To explicitly describe the action of $\G$ on $\K{n}$ consider the {\it generic octonions}
$$Z_i= \matr{z_{i1}}{(z_{i2}, z_{i3}, z_{i4})}{(z_{i5}, z_{i6}, z_{i7})}{z_{i8}}\in\OO(\K{n})$$%
for $1\leq i\leq n$. Given $g\in\G$, denote by $g\bullet Z_i$ the octonion
$$\matr{g z_{i1}}{(g z_{i2}, g z_{i3}, g z_{i4})}{(g z_{i5}, g z_{i6}, g z_{i7})}{g z_{i8}}\in\OO(\K{n}).$$%
For any commutative algebra $\algA$, the action of $\G$ on $\OO$ extends for $\OO(\algA)$ by $\algA$-linearity. In particular, $\G$ acts on $\OO(\K{n})$. It is easy to see that
\begin{eq}\label{eq_action}
g\bullet Z_i = g^{-1}Z_i,
\end{eq}%
where $g^{-1}Z_i$ stands for the action of $g^{-1}$ on the octonion $Z_i\in \OO(K_n)$.

The algebra of {\it $\G$-invariants of several octonions} ({\it octonion $G_2$-invariants}, for short) is
$$\K{n}^{\G}=\FF[\OO^n]^{\G}=\{f\in \FF[\OO^n]\,|\,gf=f \text{ for all }g\in\G\}.$$
In other words,
$$\K{n}^{\G}=\{f\in \FF[\OO^n]\,|\,f(g\un{a})=f(\un{a}) \text{ for all }g\in\G,\; \un{a}\in \OO^n\}.$$

Similarly we can define $\FF[\OO_0^n]^{\G}$, since $\OO_0\subset \OO$ is invariant with respect to $\G$-action.  Namely, the coordinate ring of the affine variety $\OO_0^n$ is $\KO{n}=\FF[\OO_0^n]=\FF[z_{ij}\,|\,1\leq i\leq n,\; 1\leq j\leq 7]$. The {\it generic traceless octonions} are
$$X_i= \matr{z_{i1}}{(z_{i2}, z_{i3}, z_{i4})}{(z_{i5}, z_{i6}, z_{i7})}{-z_{i1}}.$$
The analogue of formula~\Ref{eq_action} also holds for the generic traceless octonions, namely, $g\bullet X_i = g^{-1}X_i$ for all $g\in \G$ and $1\leq i\leq n$. The algebra of {\it $\G$-invariants of several traceless octonions} is $\KO{n}^{\G}$.

\subsection{Separating sets}\label{subsection_separating}

Consider a finite dimensional vector space $\spaceV$ and a linear group $G<\GL(\spaceV)$. In 2002 Derksen and Kemper~\cite{derksen2002computational} (see~\cite{derksen2002computationalv2} for the second edition) introduced the notion of separating invariants as a weaker concept than generating invariants. Given a subset $S$ of $\FF[\spaceV]^G$ and $u,v$ of $\spaceV$, we write $S(u)\neq S(v)$ if there exists an invariant $f\in S$ with  $f(u)\neq f(v)$. In this case we say that $u,v$ {\it are separated by $S$}. If $u,v\in \spaceV$ are separated by $\FF[\spaceV]^G$, then we say that they {\it are separated}. A subset $S\subset \FF[\spaceV]^G$ of the invariant ring is called {\it separating} if for any $u, v$ from $\spaceV$ that are separated we have that they are separated by $S$.  It is well-known that there always exists a finite separating set (see~\cite{derksen2002computational}, Theorem 2.3.15). We say that a separating set is minimal if it is minimal w.r.t.~inclusion.  Obviously, any generating set is also separating. 
Minimal separating sets and upper bounds on degrees of elements of a separating set for different actions were constructed in~\cite{Cavalcante_Lopatin_1, DM5, domokos2017, Domokos20, Domokos20Add, dufresne2014, Ferreira_Lopatin_2023,   kaygorodov2018, Kemper_Lopatin_Reimers_2022,     kohls2013, Lopatin_Reimers_1}.

\subsection{Known results}\label{section_known}

Denote by $\algZ{n}$ the non-associative $\FF$-algebra generated by the generic octonions $Z_1,\ldots,Z_n$ and $1_{\OO}$. Any product of the generic octonions is called a word of $\algZ{n}$. The unit $1_{\OO}\in \algZ{n}$ is called the empty word. Note that for every $A,B\in \algZ{n}$ we have
\begin{eq}\label{eq_algZ}
\tr(gA)=\tr(A),\; n(gA)=n(A),\; g(AB)=(gA) (gB).
\end{eq}%

\begin{lemma}\label{lemma_inv}
\begin{enumerate}
\item[(a)] The trace of any (non-associative) product of $X_1,\ldots,X_n$ and $n(X_i)$ belong to  $\KO{n}^{\G}$.

\item[(b)] The trace of any (non-associative) product of $Z_1,\ldots,Z_n$ and $n(Z_i)$ belong to  $\K{n}^{\G}$.

\item[(c)] The trace of any (non-associative) product of $Z_1,\ldots,Z_n,\ov{Z_1},\ldots,\ov{Z_n}$ belong to  $\K{n}^{\G}$.
\end{enumerate}
\end{lemma}
\begin{proof} Let $w=w(Z_1,\ldots,Z_n)$ be some (non-associative) product of $Z_1,\ldots,Z_n$. Given $g\in \G$, equalities~\Ref{eq_action}, \Ref{eq_algZ} imply that $g \tr(w) =  \tr(w(g\bullet Z_1,\ldots,g\bullet Z_n))=  \tr(w(g^{-1} Z_1,\ldots,g^{-1} Z_n))=\tr(g^{-1} w) = \tr(w)$. The case of $n(Z_i)$ is considered similarly. Part (b) is proven. The proof of part (a) is the same. Part (c) follows from part (b) and formulas
$$\tr(\ov{a})=\tr(a),\quad n(\ov{a})=n(a),\quad \tr(\ov{a}b)=\tr(a)\tr(b)-\tr(ab)$$
for all $a,b\in\OO$.
\end{proof}

In case $\FF=\QQ$ for every $A_1,\ldots,A_4\in \algZ{n}$ denote by $Q'(A_1,A_2,A_3,A_4)$ the  complete skew symmetrization of $\tr(((A_1A_2)A_3)A_4)$ with respect to its arguments, i.e.,
$$Q'(A_1,A_2,A_3,A_4) = \frac{1}{24} \sum_{\si\in \Sym_4}(-1)^{\si} \tr(((A_{\si(1)} A_{\si(2)}) A_{\si(3)}) A_{\si(4)}).$$%
In~\cite{zubkov2018} it was shown that all coefficients of $Q'(X_1,X_2,X_3,X_4)$ belong to $\ZZ[\frac{1}{2}]$. Lemma~\ref{lemmaQ} (see below) implies that all coefficients of $Q'(Z_1,Z_2,Z_3,Z_4)$ also belong to $\ZZ[\frac{1}{2}]$. Thus $Q'(A_1,A_2,A_3,A_4)$ is well-defined over an arbitrary field of odd characteristic.

In case $\Char{\FF}\neq2$ the algebra of invariants
\begin{enumerate}
\item[$\bullet$] $\KO{n}^{\G}$ is generated by $\tr(X_i X_j)$, $\tr((X_i X_j) X_k)$, $Q'(X_i, X_j, X_k, X_l)$;

\item[$\bullet$] $\K{n}^{\G}$ is generated by $\tr(Z_i)$, $\tr(Z_i Z_j)$, $\tr((Z_i Z_j) Z_k)$, $Q'(Z_i, Z_j, Z_k, Z_l)$
\end{enumerate}
for all $1\leq i,j,k,l\leq n$ (see Corollary 9 of~\cite{zubkov2018}, Section 1 of~\cite{zubkov2018}).

\subsection{New results}\label{section_new}

Denote by $\setO{n}$ the set 
$$n(X_i);\;\; \tr((\cdots ((X_{i_1} X_{i_2})X_{i_3})\cdots )X_{i_k}), \;\; i_1<\cdots<i_k,\;\; k>1,$$
for all $1\leq i,i_1,\ldots,i_k\leq n$, and by $\set{n}$ the set
$$n(Z_i);\;\; \tr((\cdots ((Z_{i_1} Z_{i_2})Z_{i_3})\cdots )Z_{i_k}), \;\; i_1<\cdots<i_k,\;\; k>0,$$
for all $1\leq i,i_1,\ldots,i_k\leq n$. Given $1\leq k\leq n$, denote by $\setO{n}^{(k)}$ and $\set{n}^{(k)}$ (respectively) the subset of $\setO{n}$ and $\set{n}$ (respectively) of elements of degree less or equal to $k$. 

In case $\Char{\FF}=2$ generators for the algebras $\KO{n}^{\G}$ and $\K{n}^{\G}$ are not known. We introduce the algebra of {\it trace $\G$-invariants of octonions} $T_{n}\subset K_{n}^{\G}$, i.e., the algebra $T_{n}$ is generated by $n(Z_1),\ldots,n(Z_n)$  and the traces of all (non-associative) products of $Z_1,\ldots,Z_n$ (see Lemma~\ref{lemma_inv}). In case $\Char{\FF}\neq2$ we obviously have that $T_{n}=K_{n}^{\G}$. We obtain the following results:
\begin{enumerate}
\item[$\bullet$] $\set{n}^{(4)}$ is a minimal (w.r.t.~inclusion) generating set for $\K{n}^{\G}$ in case $\Char{\FF}\neq2$ (see Proposition~\ref{prop_MGS});

\item[$\bullet$] $\set{n}^{(4)}$ is a minimal (w.r.t.~inclusion) separating set for $\K{n}^{\G}$ in case $\Char{\FF}\neq2$ (see Proposition~\ref{prop_MSS});

\item[$\bullet$] $T_{n}$ is minimally generated by $\set{n}$ in case $\Char{\FF}=2$ (see Theorem~\ref{theo_MGS_T});
 
\item[$\bullet$] $\set{n}^{(8)}$ is a separating set for $\K{n}^{\G}$ in case $\Char{\FF}=2$ (see Theorem~\ref{theo_separ2}).
\end{enumerate}

\section{Auxiliaries}\label{section_def}

\subsection{Indecomposable invariants}\label{subsection_ind}

Denote by $\FF\{\XX\}_n$ the free non-associative and non-commutative unital $\FF$-algebra with free generators $x_1,\ldots,x_n$, which are called letters. A word $w$ is a non-empty product of letters. The number of letters  in $w$ is the degree $\deg(w)$ of $w$. The degree of $w$ in $x_i$ is denoted by $\deg_{x_i}(w)$ and the total degree of $w$ is denoted by $\deg(w)$. The multidegree of a word $w$ is $\mdeg(w)=(\deg_{x_1}(w),\ldots, \deg_{x_n}(w))$. A word $w$ with $\deg_{x_i}(w)\leq 1$ for all $i$ is called multilinear. An element $f=\tb{\sum_i}\al_i w_i$ of  $\FF\{\XX\}_n$, where $\al_i\in\FF$ and $w_i$ is a word, is $\NN$-homogeneous ($\NN^n$-homogeneous, respectively) if there exists $d$ ($\De\in \NN^n$, respectively) such that  $\deg(w_i)=d$ ($\mdeg(w_i)=\De$, respectively) for all $i$, where $\NN$ stands for non-negative integers. Define homomorphisms of $\FF$-algebras $\phi_0: \FF\{\XX\}_n \to \algX{n}$ and $\phi: \FF\{\XX\}_n \to \algZ{n}$ by $x_i\to X_i$ and $x_i\to Z_i$ (respectively) for all $i$.
In other words, for $f=f(x_1,\ldots,x_n)\in\FF\{\XX\}_n$ we have $\phi(f)=f(Z_1,\ldots,Z_n)\in\algZ{n}$. We write $x_{i_1}\mult \cdots \mult x_{i_k}$ for some  non-associative product of $x_{i_1},\ldots,x_{i_k}$. Similar notation we use for non-associative product in $\algZ{n}$.

For $f\in \K{n}$ denote by $\deg(f)$ its degree and by $\mdeg(f)$ its multidegree, i.e.,
$\mdeg(f)=(t_1,\ldots,t_n)$, where $t_i$ is the total degree of the polynomial $f$ in $z_{ij}$, $1\leq j\leq 8$, and $\deg(f)=t_1+\cdots+t_n$. For $f\in\KO{n}$ the degree and multidegree are defined as above. It is well-known that the algebras $\KO{n}^{\G}$ and $\K{n}^{\G}$ have $\NN$-gradings by degrees and $\NN^n$-gradings by multidegrees. 

Consider an $\NN^n$-graded unital (possibly, non-associative) algebra $\algA$ with the component of degree zero equal to $\FF$. Denote by $\algA^{+}$ the subalgebra generated by homogeneous elements of positive degree. A set $\{a_i\} \subseteq \algA$ is a minimal (by inclusion) $\NN^n$-homogeneous generating set
(m.h.g.s.) of $\algA$ as a unital algebra~if and only if the $a_i$'s are $\NN^n$-homogeneous and $\{\ov{a_i}\}\cup\{1\}$ is a basis of the vector space $\ov{\algA}={\algA}/{(\algA^{+})^2}$. 
We say that an element $a\in \algA$ is {\it
decomposable} and we write $a\equiv0$ if $a\in (\algA^{+})^2$. In other words, a decomposable
element is equal to a polynomial in elements of strictly less degree. Therefore,  the largest degree 
of indecomposable elements of  $\algA$  is equal to the least upper bound for the degrees of elements of a m.h.g.s.~for $\algA$.

\subsection{One-parameter subgroups of $\G$}\label{subsection_1gr}

Consider a finite dimensional vector space $\spaceV$ and a linear (closed) group $G<\GL(\spaceV)$. For a point $v\in \spaceV$ and for a one-parameter subgroup $\theta: \FF^{\times}\to G$ we have $\theta(t)v=\sum_{i\in I(v)} t^i v^{(i)}$ for all $t\in \FF^{\times}$, where $I(v)=\{ i\in\mathbb{Z}\mid v^{(i)}\neq 0 \}$. Following \cite{kempf_1978} we
say that $\lim_{t\to 0}\theta(t)v$ exists if and only if $I(v)$ consists of non-negative integers. Then $\lim_{t\to 0}\theta(t)v =0$ if and only if $I(v)$ consists of positive integers only, otherwise $\lim_{t\to 0}\theta(t)v=v^{(0)}$. It is clear that if $\lim_{t\to 0}\theta(t)v$ exists, then it is contained in $\overline{Gv}$. Indeed, if $f$ is a polynomial function on $\spaceV$, that vanishes on the $G$-orbit of $v$, then $h(t)=f(\theta(t)v)$ is a polynomial in $t$, such that $h(t)=0$ for any $t\neq 0$. Since $\FF$ is infinite, $h(t)$ is identically zero, i.e., $h(0)=
f(v^{(0)})=0$.

Given $\un{\la}\in\ZZ^3$ with $\la_1+\la_2+\la_3=0$, the {\it standard} one-parameter subgroup $\theta_{\un{\la}}$ of $\G$ is defined by 
$$\theta_{\un{\la}}(t)e_i=e_i,\quad  
\theta_{\un{\la}}(t)\uu_j= t^{\la_j}\uu_j,\quad
\theta_{\un{\la}}(t)\vv_j= t^{-\la_j}\vv_j, $$
for all $i= 1,2$ and $1\leq j\leq3$. 

\section{Minimal generating and separating sets}\label{section_MGS}

In this section we write $\tr(i_1,\ldots,i_k)$ for $\tr((\cdots((Z_{i_1} Z_{i_2}) Z_{i_3}) \cdots )Z_{i_k})$, where $1\leq i_1,\ldots,i_k\leq n$. The following lemma can be proven by straightforward calculations.

\begin{lemma}\label{lemmaQ} Assume that $\Char{\FF}\neq2$. Then
$$Q'(Z_1,Z_2,Z_3,Z_4) = 
\tr(1234) + \frac{1}{2} ( - \tr(1) \tr(2) \tr(3) \tr(4)  + $$
  $$ - \tr(1) \tr(234) - \tr(2) \tr(134) -
     \tr(3) \tr(124) - \tr(4) \tr(123) + $$
  $$ - \tr(12) \tr(34) + \tr(13) \tr(24) -
     \tr(14) \tr(23)  +$$
   $$\tr(1) \tr(2) \tr(34) + \tr(1) \tr(4) \tr(23) +
     \tr(2) \tr(3) \tr(14) + \tr(3) \tr(4) \tr(12)).$$
\end{lemma}
\medskip

Recall that the definition of $T_n$ was given in Section~\ref{section_new}.

\begin{lemma}\label{lemma1}
Let $w\in\FF\{\XX\}_n$ be a word.
\begin{enumerate}
\item[1.] If $w$ is not multilinear and $\deg(w)>2$, then $\tr(w(Z_1,\ldots,Z_n))\equiv0$ in $T_{n}$.

\item[2.] If $w$ is multilinear and $w$ is a product of letters $x_{i_1},\ldots, x_{i_k}$ for $1\leq i_1<\cdots<i_k\leq n$, then $$\tr(w(Z_1,\ldots,Z_n))\equiv\pm\tr(i_1,\ldots ,i_k)\;\text{ in } \;T_{n}.$$

\item[3.] For all $1\leq i_1<\cdots<i_k\leq n$ with $k\geq 3$ and every permutation $\si\in \Sym_k$ we have $$\tr(i_{\si(1)},\ldots ,i_{\si(k)})\equiv (-1)^{\si} \tr(i_1,\ldots ,i_k)\;\text{ in } \;T_{n}.$$
\end{enumerate}
\end{lemma}
\begin{proof} Combining \Ref{eq3} and \Ref{eq5} we obtain that
$$
a(a'b) + a'(ab)=(\tr(a) a' + \tr(a') a  +\tr(aa') - \tr(a)\tr(a'))b
$$
for all $a,a',b\in\OO$. Since $\FF$ is infinite, the same equality holds for the generic octonions. We multiply it from the left and from the right by the generic octonions and then apply the trace. Since the trace is a linear function, we obtain that
\begin{eq}\label{eq_trAlt1}
\tr(C_1\mult \cdots \mult C_r\mult (A(A'B)) \mult C_{r+1}\mult \cdots \mult C_s) \equiv -\tr(C_1\mult\cdots \mult  C_r\mult (A'(AB)) \mult C_{r+1}\mult \cdots \mult C_s)
\end{eq}%
for all products of the generic octonions $A,A',B, C_1,\ldots,C_s$ with $0\leq r\leq s$ and $s\geq 0$. Similarly, we obtain that
\begin{eq}\label{eq_trAlt2}
\tr(C_1\mult \cdots \mult C_r\mult ((BA) A') \mult C_{r+1}\mult \cdots \mult C_s) \equiv -\tr(C_1\mult\cdots \mult  C_r\mult ((BA') A) \mult C_{r+1}\mult \cdots \mult C_s)
\end{eq}%
In the same manner as above, \Ref{eq2} and \Ref{eq3} imply that
\begin{eq}\label{eq_trAA}
\tr(C_1\mult \cdots \mult C_r\mult (A^2) \mult C_{r+1}\mult \cdots \mult C_s)\equiv0,
\end{eq}
\vspace{-0.5cm}
\begin{eq}\label{eq_trAB}
\tr(C_1\mult \cdots \mult C_r\mult (AA') \mult C_{r+1}\mult \cdots \mult C_s) \equiv -\tr(C_1\mult\cdots \mult  C_r\mult (A'A) \mult C_{r+1}\mult \cdots \mult C_s),
\end{eq}%
where in both cases $0\leq r\leq s$ and $s> 0$. We claim that

\begin{eq}\label{eq_claim}
\begin{array}{c}
\text{If $W=Z_{i_1}\mult \cdots \mult Z_{i_k}$ is a product of generic octonions} \\
\text{\tb{ where $1\leq i_1,\ldots,i_k\leq n$},}\\
\text{then $\tr(W)\equiv\pm\tr(i_{\si(1)},\ldots,i_{\si(k)})$ for some $\sigma\in \Sym_k$}.\\
\end{array}
\end{eq}

\noindent{}Assume that claim~\Ref{eq_claim} does not hold. Then there exists $\tau\in \Sym_k$ and \tb{the maximal $2\leq r< k$} such that some product $W'=Z_{i_{\tau(1)}}\mult \cdots \mult Z_{i_{\tau(k)}}$ satisfies the following conditions: $\tr(W)\equiv\pm\tr(W')$ and $$W'=C_1\mult \cdots \mult  (U (V_1 V_2)) \mult \cdots \mult C_s
\;\text{ or }\; W'=C_1\mult \cdots \mult  (VU) \mult \cdots \mult C_s, \text{ where}
$$
\begin{enumerate}
\item[$\bullet$] $U=(\cdots((Z_{j_1} Z_{j_2}) Z_{j_3}) \cdots )Z_{j_r}$ for some $1\leq j_1,\ldots, j_r\leq n$,  

\item[$\bullet$]$V,V_1,V_2$ are some products of generic octonions, 

\item[$\bullet$] $C_1,\ldots,C_s$ are generic octonions  with $s\geq0$. 
\end{enumerate}
By~\Ref{eq1} and~\Ref{eq_trAB}, we can assume that $W'=C_1\mult \cdots \mult  (U (V_1 V_2)) \mult \cdots \mult C_s$. Consequently, applying equivalence~\Ref{eq_trAlt1} and equivalence~\Ref{eq1} or~\Ref{eq_trAB}, we obtain that 
$$\tr(C_1\mult \cdots \mult  (U (V_1 V_2)) \mult \cdots \mult C_s)\equiv 
- \tr(C_1\mult \cdots \mult  (V_1 (U V_2)) \mult \cdots \mult C_s)\equiv$$
$$\pm \tr(C_1\mult \cdots \mult  ((U V_2)V_1) \mult \cdots \mult C_s).$$
If $V_2$ is a product of more than one generic octonions, then $V_2=V_2' V_2''$ for some products $V_2'$, $V_2''$ of generic octonions and we repeat the reasoning for $C_1\mult \cdots \mult  (U (V'_2 V''_2)) \mult \cdots \mult C_s$, \tb{and so on}. \tb{Finally}, we can assume that $V_2=Z_j$ for some $j$; a contradiction to the \tb{maximality of r}.

\tb{Equivalences \Ref{eq_trAlt2} and~\Ref{eq_trAB} imply that part 3 is valid for $1\leq i_1,\ldots, i_k\leq n$, where $k\geq3$. This fact together with claim~\Ref{eq_claim} imply part 2. Similarly, this fact together with claim~\Ref{eq_claim} and formula~\Ref{eq_trAA} imply part 1.}

\end{proof}



\begin{prop}\label{prop_MGS}
In case $\Char{\FF}\neq2$ the algebra of invariants $\K{n}^{\G}$ is minimally generated by $\set{n}^{(4)}$.
\end{prop}
\begin{proof}
The description of generators for $\K{n}^{\G}$ from~\cite{zubkov2018} (see Section~\ref{section_known} for the details) together with Lemmas~\ref{lemmaQ},~\ref{lemma1} and formula~\Ref{eq_n} imply that the set $\set{n}^{(4)}$ generates the algebra $\K{n}^{\G}$. By Corollary~1 of~\cite{zubkov2018} and formula~\Ref{eq_n}, the invariants
  $$\tr(i),\; n(Z_i),\; \tr(12),\; \tr(13),\; \tr(23),\; \tr(123),$$
where $1\leq i\leq 3$, are algebraically independent over $\FF$. Thus the required statement is proven for $n\leq3$. 

Assume $n\geq4$. Thus $\set{n}^{(4)}\backslash\{f\}$ is not a generating set for any $f\in \set{n}^{(4)}$ with $\deg(f)\neq 4$.

Assume that $\set{n}^{(4)}\backslash\{\tr(1234)\}$ is a generating set. Then $\tr(1234)$ is a linear combination of $\tr(12)\tr(34)$, $\tr(13)\tr(24)$, $\tr(14)\tr(23)$ and products containing $\tr(i)$ for some $1\leq i\leq 4$. Considering substitutions $Z_1\to  \vv_1$, $Z_2\to \vv_2$, $Z_3\to \vv_3$, $Z_4\to e_1-e_2$ and using equalities $\tr(((\vv_1\vv_2)\vv_3) (e_1-e_2))=-1$ and $\tr(\vv_i (e_1-e_2))=0$ for $1\leq i\leq 3$, we obtain a contradiction. The proposition is proven. 
\end{proof}

\begin{remark}\label{remark3}
\begin{enumerate}
\item[1.] By~\Ref{eq_n}, in the formulation of Proposition~\ref{prop_MGS} we can replace $n(Z_i)$ by $\tr(Z_i^2)$ for all $1\leq i\leq n$.

\item[2.] It easily follows from \tb{the proof of} Proposition~\ref{prop_MGS} (see also Section 1 of~\cite{zubkov2018}) that $\KO{n}^{\G}$ is minimally generated by $\setO{n}^{(4)}$ \tb{when $\Char{\FF}\neq2$}.
\end{enumerate}
\end{remark}

\begin{prop}\label{prop_MSS}
Assume $\Char{\FF}\neq2$. Then $\setO{n}^{(4)}$ and $\set{n}^{(4)}$ are minimal separating sets for  $\KO{n}^{\G}$ and $\K{n}^{\G}$ (respectively) for all $n>0$.
\end{prop}
\begin{proof} By Proposition~\ref{prop_MGS} and Remark~\ref{remark3}, the sets $\setO{n}^{(4)}$ and $\set{n}^{(4)}$ are separating for $\KO{n}^{\G}$ and $\K{n}^{\G}$ (respectively). For $a=0$, $b=\uu_1+\vv_1$ we have $\tr(a)=n(a)=\tr(b)=0$, but $n(b)=\tb{-1}$. For $a=0$, $b=e_1$ we have $\tr(a)=n(a)=n(b)=0$, but $\tr(b)=1$. Hence, $\set{1}$ is a minimal separating set for $\K{1}^{\G}$. Claims 1, 2, 3 (see below) imply that $\setO{n}^{(4)}$ is a {\it minimal} separating set for $\K{0,n}^{\G}$.  Therefore,   $\set{n}^{(4)}$ is also a minimal separating set for $\K{n}^{\G}$.

\medskip
\noindent{}{\it Claim 1.} Let $n=2$. Then $\setO{2}\backslash \{\tr(X_1 X_2)\}$ is not separating $\KO{2}^{\G}$.
\smallskip

To prove this claim consider $\un{a}=(0,0)$ and $\un{b}=(\uu_1,\vv_1)$ from $\OO_0^2$. Then $\tr(a_1 a_2)\neq \tr(b_1 b_2)$.

\medskip
\noindent{}{\it Claim 2.} Let $n=3$. Then $\setO{3}\backslash \{\tr((X_1 X_2) X_3)\}$ is not separating for $\KO{3}^{\G}$.
\smallskip

To prove this claim we consider $\un{a}=(0,0, 0)$ and $\un{b}=(\vv_1,\vv_2,\vv_3)$ from $\OO_0^3$. Then $\tr(a_i a_j)=\tr(b_i b_j)=0$ for all $1\leq i<j\leq 3$, but $\tr(a_1 a_2 a_3)\neq \tr(b_1 b_2 b_3)$.

\medskip
\noindent{}{\it Claim 3.} Let $n=4$. Then $\setO{4}\backslash \{\tr(((X_1 X_2) X_3) X_4)\}$ is not separating for $\KO{4}^{\G}$.
\smallskip

To prove this claim we consider $\un{a}=(\uu_1,\vv_1,c,\uu_2)$ and $\un{b}=(\uu_1,\vv_1,c,-\vv_2)$ from $\OO_0^4$, where $c=e_1+\uu_2 -\vv_2 - e_2$. Then $\tr(a_i a_4)=\tr(b_i b_4)$ for $1\leq i\leq 3$ and $\tr((a_i a_j) a_4) = \tr((b_i b_j) b_4)$ for $1\leq i<j \leq 3$, but $\tr(((a_1 a_2) a_3) a_4) = 0$ and $\tr(((b_1 b_2) b_3) b_4)=-1$.
\end{proof}

\section{Trace invariants}\label{section_trace}

The group $\GL_2=\GL_2(\FF)$ acts on $M_2^n=M_2(\FF)^{\oplus n}$ diagonally by conjugation. The coordinate ring $\FF[M_2^n]=\FF[z_{i1},z_{i2},z_{i5},z_{i8}\,|\,1\leq i\leq n]$ is also a $\GL_2$-module, where the {\it generic matrices} are
$$\widehat{Z}_i=\matr{z_{i1}}{z_{i2}}{z_{i5}}{z_{i8}}.$$
Note that we consider $\FF[M_2^n]$ as a subalgebra of $K_n$. In~\cite{DKZ02} it was shown that 
\begin{eq}\label{eq_mgs_matrix_inv}
\det(\widehat{Z}_i),\;\; 1\leq i\leq n;\;\; \tr(\widehat{Z}_{i_1}\cdots \widehat{Z}_{i_k}), \;\; 1\leq i_1<\cdots<i_k\leq n,\;\; k>0,
\end{eq}%
is a minimal generating set for $\FF[M_2^n]^{\GL_2}$, where $k\leq 3$ in case $\Char{\FF}\neq2$. In particular, all elements from  set~\Ref{eq_mgs_matrix_inv} are indecomposable. Note that a minimal separating set for $\FF[M_2^n]^{\GL_2}$ was obtained in~\cite{kaygorodov2018}.

Define a surjective homomorphism of $\FF$-algebras $\Psi:\K{n}\to \FF[M_2^n]$ as
follows: $z_{i3}\to 0$, $z_{i4}\to 0$, $z_{i6}\to 0$, $z_{i7}\to 0$ for all $i$. We can naturally extend $\Psi$ to the linear map $\widehat{\Psi}:\OO(\K{n}) \to \OO(\FF[M_2^n])$ by 
$$\widehat{\Psi}\matr{f_1}{(f_2,f_3,f_4)}{(f_5,f_6,f_7)}{f_8} = 
\matr{\Psi(f_1)}%
{(\Psi(f_2),\Psi(f_3),\Psi(f_4))}%
{(\Psi(f_5),\Psi(f_6),\Psi(f_7))}%
{\Psi(f_8)}$$
for $f_1,\ldots,f_8\in K_n$.

For an associative commutative $\FF$-algebra $\algA$ define a map $\mathcal{F}: M_2(\algA) \to \OO(\algA)$ by 
$$\matr{a_1}{a_2}{a_3}{a_4} \to \matr{a_1}{(a_2,0,0)}{(a_3,0,0)}{a_4}$$
for $a_1,\ldots,a_4\in\algA$. It is easy to see that $\mathcal{F}$ is an injective homomorphism of algebras \tb{preserving the trace}, since $(a,0,0)\times (b,0,0)=0$ for all $a,b\in \algA$. In what follows, we consider $\algA=K_n$. \tb{Since the homomorphism $\widehat{\Psi}$ commutes with the trace and the norm, we obtain the following lemma.}

\begin{lemma}\label{lemma_psi}
For all $1\leq i,i_1,\ldots,i_k\leq n$ we have
\begin{enumerate}
\item[(a)] $\widehat{\Psi}(\cdots ((Z_{i_1} Z_{i_2})Z_{i_3})\cdots )Z_{i_k}) = \mathcal{F}(\widehat{Z}_{i_1} \cdots \widehat{Z}_{i_k})$;

\item[(b)] $\Psi(\tr((\cdots ((Z_{i_1} Z_{i_2})Z_{i_3})\cdots )Z_{i_k})) = \tr(\widehat{Z}_{i_1}\cdots\widehat{Z}_{i_k})$;

\item[(c)] $\Psi(n(Z_i)) = \det(\widehat{Z}_i)$.
\end{enumerate}
\end{lemma}

Lemma~\ref{lemma_inv}, Lemma~\ref{lemma_psi} and the description of generators of $\FF[M_2^n]^{\GL_2}$ imply that
\begin{eq}\label{eq_psi_image}
\FF[M_2^n]^{\GL_2} \subset \Psi(\K{n}^{\G}),
\end{eq}%
where we have the equality in case $\Char{\FF}\neq2$.

\begin{theo}\label{theo_MGS_T}
In case $\Char{\FF}=2$ the algebra of trace $\G$-invariants  $T_{n}$ is minimally generated by $\set{n}$.
\end{theo}
\begin{proof}By Lemma~\ref{lemma1} and formula~\Ref{eq_n}, the algebra $T_{n}$ is generated by $\set{n}$. To show that $\set{n}$ is a minimal generating set, it is enough to prove that every element $f\in \set{n}$ is indecomposable in $T_n$. Assume the contrary. If $f=\tr((\cdots ((Z_{i_1} Z_{i_2})Z_{i_3})\cdots )Z_{i_k})$
\tb{from $S_n$} \tb{were decomposable in $T_{n}$}, then by parts (b), (c) of Lemma~\ref{lemma_psi}, $\Psi(f)=\tr(\widehat{Z}_{i_1}\cdots \widehat{Z}_{i_k})$ \tb{would be} decomposable in $\FF[M_2^n]^{\GL_2}$; a contradiction. Similarly, if $f=n(Z_i)$ \tb{were decomposable in $T_{n}$}, then $\Psi(f)=\det(\widehat{Z}_{i})$ \tb{would be} decomposable in $\FF[M_2^n]^{\GL_2}$; a contradiction.
\end{proof}

\section{Subalgebras of $\OO$ of low dimension}\label{section_subalgebras}

\tb{The group $\G$ acts naturally on the set of subalgebras of $\OO$. For a subalgebra $\algA$ of $\OO$ we denote by $[\algA]$ the $\G$-orbit of $\algA$ and we say that $[\algA]$ is the equivalence class of $\algA$. Obviously, all algebras in $[\algA]$ are isomorphic to $\algA$. Denote by $\Omega(\OO)$ the set of $\G$-orbits (i.e., equivalence classes) in the set of subalgebras of $\OO$.}
Since all algebras from a given equivalence class $\mathfrak{A}\in \Omega(\OO)$ have the same dimension, we call it the \emph{dimension} of $\mathfrak{A}$. A set of (linearly independent) octonions is said to be a \emph{basis} of $\mathfrak{A}$, provided they form a basis of \tb{an} algebra from $\mathfrak{A}$. An equivalence class   $\mathfrak{A}\in\Omega(\OO)$ is called {\it closed} if there exists a subalgebra $\algA$ of $\OO$ with $[\algA]=\mathfrak{A}$ and there is an $\FF$-basis $a_1,\ldots,a_n$ of $\algA$ such that the $\G$-orbit of $(a_1,\ldots,a_n)$ is closed in $\OO^n$. 
More details on the definition of a closed equivalence class can be found in Remark~\ref{remak_closed} (see below). Denote by
$$\MM=\matr{\ast}{(\ast,0,0)}{(\ast,0,0)}{\ast}\;\;\text{ and }\;\; \SE=\matr{\ast}{(\ast,\ast,0)}{(\ast,0,\ast)}{\ast},$$
respectively, the subalgebra of {\it quaternions} and {\it sextonions} of $\OO$, respectively, where the term {\it sextonions} was introduced in~\cite{Landsberg_2006}. Note that $\mathcal{F}:M_2(\FF)\to\MM$ is an isomorphism of $\FF$-algebras (see Section~\ref{section_trace} for the details).

The main result of this section is the following statement. 

\begin{prop}\label{prop_subalg}
Assume $\Char{\FF}=2$ and an equivalence class $\mathfrak{A}\in\Omega(\OO)$ has dimension $d\leq3$. Then one of the following sets is a basis for $\mathfrak{A}$: 
\begin{enumerate}
\item[$d=1${\rm:}] $\{1_{\OO}\}$, $\{\uu_1\}$, $\{e_1\}$;

\item[$d=2${\rm:}] $\{1_{\OO},\uu_1\}$, $\{\uu_1,\vv_2\}$, $\{e_1,\uu_1\}$, $\{e_1,\vv_1\}$, $\{e_1,e_2\}$;

\item[$d=3${\rm:}] $\{1_{\OO},\uu_1,\vv_2\}$, $\{e_1,e_2,\uu_1\}$, $\{e_1,\uu_1,\vv_2\}$, \tb{$\{\uu_1,\vv_2,\vv_3\}$}.
\end{enumerate}
\end{prop}

Note that we do not require for a subalgebra of $\OO$ to be unital. The proof of Proposition~\ref{prop_subalg} will be given in \tb{a} series of propositions and lemmas, which are interesting on their \tb{own}.

\begin{prop}[Proposition 3.4 of \cite{LZ_1}]\label{prop_ThI}
For each $a\in\OO$ there exists $g\in\G$ such that $g a$ is a canonical octonion of one of the following types: 
\begin{enumerate}
\item[{\rm(D)}] $\matr{\al_1}{\zero}{\zero}{\al_8}$,

\item[{\rm($\mathrm{K}_1$)}] $\matr{\al_1}{(1,0,0)}{\zero}{\al_1}$,
\end{enumerate}
for some $\al_1,\al_8\in\FF$.  Moreover, these canonical octonions are unique modulo permutation $\al_1 \leftrightarrow \al_8$ for type {\rm(D)}.
\end{prop}

\begin{prop}[Corollary 5.2 of \cite{LZ_1}]\label{prop_ThII}
Assume $\Char{\FF}=2$. For each $(a,b)\in\OO_0^2$ there exists $g\in\G$ such that $g (a,b)$ is a pair of one of the following types: 
\begin{enumerate}
\item[{\rm (EE)}] $(\al_1 1_{\OO},\be_1 1_{\OO})$,

\item[{\rm($\rm E K_1$)}] 
$\left(\al_1 1_{\OO}, \matr{\be_1}{(1,0,0)}{\zero}{\be_1}\right)$,

\item[{\rm($\rm K_1 E$)}] 
$\left(\matr{\al_1}{(1,0,0)}{\zero}{\al_1}, \be_1 1_{\OO} \right)$,

\item[{\rm($\rm K_1 L_1$)}] 
$\left(\matr{\al_1}{(1,0,0)}{\zero}{\al_1}, \matr{\be_1}{(\be_2,0,0)}{\zero}{\be_1} \right)$ with $\be_2\neq 0$,

\item[{\rm($\rm K_1 L_1^{\top}$)}] 
$\left(\matr{\al_1}{(1,0,0)}{\zero}{\al_1}, \matr{\be_1}{\zero}{(\be_5,0,0)}{\be_1} \right)$ with $\be_5\neq 0$,

\item[{\rm($\rm K_1 M_1$)}] 
$\left(\matr{\al_1}{(1,0,0)}{\zero}{\al_1}, \matr{\be_1}{(0,1,0)}{\zero}{\be_1} \right)$,

\item[{\rm($\rm K_1 M_1^{\top}$)}] 
$\left(\matr{\al_1}{(1,0,0)}{\zero}{\al_1}, \matr{\be_1}{\zero}{(0,1,0)}{\be_1} \right)$,
\end{enumerate}
where $\al_1,\be_1,\be_2,\be_5\in\FF$.
\end{prop}
\medskip

\begin{remark}[see Remarks 3.1, 3.2 and 3.3 of \cite{LZ_1}]\label{remark_Can1.2}
Assume $a=\matr{\al}{\uu}{\vv}{\be}\in\OO$. Then 
\begin{enumerate}
\item[(a)] if $\uu\neq 0$, then there exists $g\in\SL_3$ such that 
$g a = \matr{\al}{(1,0,0)}{\vv'}{\be}$, where $\vv'=(\ast,0,0)$ or $\vv'=(0,1,0)$;

\item[(b)] if $\vv\neq 0$, then there exists $g\in\SL_3$ such that 
$g a = \matr{\al}{\uu'}{(1,0,0)}{\be}$, where $\uu'=(\ast,0,0)$ or $\uu'=(0,1,0)$.

\item[(c)] there exist $g,g',g''\in\SL_3$ such that  
$$g(\uu_1,\vv_1,\uu_2,\vv_3) = (\uu_1,\vv_1,\uu_3,-\vv_2),$$
$$g'(\uu_2,\vv_2,\uu_1,\vv_3) = (\uu_2,\vv_2,\uu_3,-\vv_1),$$
$$g''(\uu_3,\vv_3,\uu_1,\vv_2) = (\uu_3,\vv_3,\uu_2,-\vv_1).$$

\item[(d)]  if $\uu=(\ga_1,\ga_2,\ga_3)$ with $\ga_2\neq0$ or $\ga_3\neq0$ and $\vv=(\de,0,0)$, then there exists $g\in\SL_3$ such that  $g a= \matr{\al}{(\ga_1, 1, 0)}{(\de, 0, 0)}{\be}$ and \tb{$g(\uu_1,\vv_1)=(\uu_1,\vv_1)$}.

\item[(e)] \tb{if $\vv=(\ga_1,\ga_2,\ga_3)$ with $\ga_2\neq0$ or $\ga_3\neq0$ and $\uu=(\de,0,0)$, then there exists $g\in\SL_3$ such that  $g a= \matr{\al}{(\de, 0, 0)}{(\ga_1, 1, 0)}{\be}$ and $g(\uu_1,\vv_1)=(\uu_1,\vv_1)$}.
\end{enumerate}
\end{remark}

\tb{The following lemma is an immediate consequence of the Cayley-Dickson doubling process (see also Section 2.1 of~\cite{Springer_Veldkamp_2000}). } Its analogue over a finite field is part (ii) of Lemma 3.3 from~\cite{Grishkov_2010}.
\begin{lemma}\label{lemma_Grishkov} Every automorphism of the $\FF$-algebra $\MM$ can be extended to an automorphism of the algebra $\OO$.
\end{lemma}

\begin{lemma}\label{lemma1_algO_2}
If $\algA\subset\OO$ is a non-zero subalgebra, then there exists $g\in\G$ such that $1_{\OO}\in g\algA$ or $\uu_1\in g\algA$ or $e_1\in g\algA$.  In particular, if $\Char{\FF}=2$ and $\algA\not\subset \OO_0$ is a non-zero subalgebra of $\OO$, then there exists $g\in\G$ such that $e_1\in g\algA$.
\end{lemma}
\begin{proof} \tb{This follows from Proposition~\ref{prop_ThI}, the known corresponding
statement for the algebra $\MM \simeq M_2(\FF)$ and Lemma~\ref{lemma_Grishkov}.}
\end{proof}

\subsection{The case of traceless subalgebra}

In this section we assume that $\Char{\FF}=2$ and $\algA\subset \OO$ is a subalgebra of traceless octonions, i.e., $\algA\subset \OO_0$.

\begin{remark}\label{remark1_algO_2}
If $a=\matr{\al}{\uu}{\vv}{\be}\in\algA$ is {\it triangular} (i.e., $\uu=\zero$ or $\vv=\zero$), and $\al\neq0$ or $\be\neq0$, then $1_{\OO}\in\algA$.  
\end{remark}
\begin{proof}
Since $\al=\be$ is non-zero, considering $a^2=\al^2 1_{\OO}$ we complete the proof.
\end{proof}

\begin{lemma}\label{lemma2_algO_2}
If $\dim\algA\geq2$, then there exists $g\in\G$ such that one of the following possibilities holds:
\begin{enumerate}
\item[(a)] $\{1_{\OO},\uu_1\}\subset g\algA$;

\item[(b)] $\{\uu_1,\vv_2\}\subset g\algA$ and $1_{\OO}\not\in g\algA$.
\end{enumerate}  
\end{lemma}
\begin{proof} By Lemma~\ref{lemma1_algO_2}, we assume that one of the following alternatives holds:

\medskip
\noindent{\bf 1.} $1_{\OO}\in\algA$. There exists $a\in\algA$ such that $\{1_{\OO},a\}$ are linearly independent. Since $\G 1_{\OO}=1_{\OO}$, by Proposition~\ref{prop_ThI}, we can assume that $a=\al e_1 + \be e_2$ or $a=\al 1_{\OO} + \uu_1$ for some $\al,\be\in\FF$. In the first case we have $\al=\be$ and  $\{1_{\OO},a\}$ are linearly dependent; a contradiction. In the second case we obtain that $\uu_1 = a - \al 1_{\OO}$ lies in $\algA$.

\medskip
\noindent{\bf 2.} $\uu_1\in\algA$ and $1_{\OO}\not\in\algA$.  There exists $b\in\algA$ such that $\{\uu_1,b\}$ are linearly independent. Consider  $g\in\G$ such that $g (\uu_1,b)=(a',b')$ is one of the pairs from Proposition~\ref{prop_ThII}. Since $\tr(a')=n(a')=0$ and $a'\neq0$, one easily sees that $a'=\uu_1$. By Remark~\ref{remark1_algO_2} and the fact that $1_{\OO}\not\in\algA$ one sees that both diagonal entries of $b'$ are equal to zero. Using the fact that $\{\uu_1,b'\}$ are linearly independent, we obtain that the pair $(\uu_1,b')$ has one of the following types:

\begin{enumerate}
\item[{\rm($\rm K_1 L_1^{\top}$)}] $b'=\be \vv_1$, where $\be\in\FF\setminus \{0\}$. Since $\uu_1 b'=\be e_1$ and $\tr(e_1)\neq0$, we obtain a contradiction.

\item[{\rm($\rm K_1 M_1$)}] $b'=\uu_2$. Since $\uu_1 b' = \vv_3$, acting by a suitable element of $\SL_3$ and using part (c) of Remark~\ref{remark_Can1.2}, we obtain case (b).

\item[{\rm($\rm K_1 M_1^{\top}$)}] $b'=\vv_2$, i.e., we have case (b).
\end{enumerate}
\end{proof}

\begin{lemma}\label{lemma3_algO_2}
If $\dim\algA\geq3$, then there exists $g\in\G$ such that one of the following possibilities holds:
\begin{enumerate}
\item[(a)] $\{1_{\OO},\uu_1,\vv_2\}\subset g\algA$;

\item[(b)]\tb{ $\{\uu_1,\vv_2,\vv_3\}\subset g\algA$} and $1_{\OO}\not\in g\algA$. 
\end{enumerate}  
\end{lemma}
\begin{proof}  By Lemma~\ref{lemma2_algO_2}, one can assume  that one of the following possibilities holds:

\medskip
\noindent{\bf 1.} $\{1_{\OO},\uu_1\}\subset \algA$. There exists $b\in\algA$ such that $\{1_{\OO},\uu_1,b\}$ are linearly independent.  Consider  $g\in\G$ such that $g (\uu_1,b)=(a',b')$ is one of the pairs from Proposition~\ref{prop_ThII}. Since $\tr(a')=n(a')=0$ and $a'\neq0$, one easily sees that $a'=\uu_1$. Let $\be'$ be the diagonal element of $b'$. Since $\G 1_{\OO}=1_{\OO}$, taking $b''=b'-\be' 1_{\OO}$ instead of $b'$, we can assume that $\{1_{\OO},\uu_1,b''\}\subset\algA$ are linearly independent and $(\uu_1,b'')$ has one of types from Proposition~\ref{prop_ThII}, where the diagonal elements of $b''$ are zeros. Consider the possible types for $(\uu_1,b'')$: 
\begin{enumerate}
\item[{\rm($\rm K_1 L_1^{\top}$)}] $\{1_{\OO},\uu_1,\be \vv_1\}\subset\algA$ for some non-zero $\be\in\FF$. Since $\uu_1 \vv_1=e_1$, we obtain a contradiction.

\item[{\rm($\rm K_1 M_1$)}] $\{1_{\OO},\uu_1,\uu_2\}\subset\algA$. Since $\uu_1 \uu_2 = \vv_3$, acting by a suitable element of $\SL_3$ from part (c) of Remark~\ref{remark_Can1.2} we obtain case (a).

\item[{\rm($\rm K_1 M_1^{\top}$)}] $\{1_{\OO},\uu_1,\vv_2\}\subset\algA$, i.e., we have case (a).
\end{enumerate}

\medskip
\noindent{\bf 2.} $\{\uu_1,\vv_2\}\subset \algA$ and $1_{\OO}\not\in \algA$. Consider $b\in\algA$ such that $\{\uu_1,\vv_2, b\}$ are linearly independent. Moreover, one can assume that $b=\matr{\be_1}{(0,\be_3,\be_4)}{(\be_5,0,\be_7)}{\be_1}$ for some $\be_i\in\FF$. Since
$$\uu_1 b = \matr{\be_5}{(\be_1,0,0)}{(0,-\be_4,\be_3)}{0}\quad \text{and}\quad  
\vv_2 b = \matr{0}{(-\be_7,0,\be_5)}{(0,\be_1,0)}{\be_3},$$%
we have $\be_3=\be_5=0$. The equality $b^2=(\be_1^2 + \be_4 \be_7)1_{\OO}$ implies that $\{\uu_1,\vv_2,b\}\subset \algA$, where the element $b=\be_1 1_{\OO} + \be_4 \uu_3 + \be_7 \vv_3$ is non-zero and $\be_1^2=\be_4\be_7$. 

\tb{Let $\be_1=0$. Then $\uu_3$ lies in $\algA$ or case (b) holds. If $\uu_3\in\algA$, then  $\{\uu_1,\vv_2,\uu_3\}\subset\algA$; thus, $\{\vv_1,\uu_2,\vv_3\}\subset\hbar\, \algA$ and part (c) of Remark~\ref{remark_Can1.2} implies that case (b) holds.}

\tb{Let $\be_1\neq 0$. Then $\be_4,\be_7\neq0$ and for $g=\de_1(0,0,\be_1/\be_7)$ from $\G$ we have $g(\uu_1,\vv_2,b)=(\uu_1+\frac{\be_1}{\be_7}\vv_2,\vv_2,\be_7 \uu_3)$. Therefore, $\{\uu_1,\vv_2,\uu_3\}\subset g\algA$ and case (b) holds (see above).}
\end{proof}

\subsection{The case of non-traceless subalgebra}\label{section_non_traceless}

In this section we assume that $\Char{\FF}=2$ and $\algA\not\subset \OO_0$ is a subalgebra of $\OO$.

\begin{lemma}\label{lemma2_algO_1}
If $\dim\algA\geq2$, then there exists $g\in\G$ such that one of the following possibilities holds:
\begin{enumerate}
\item[(a)] $\{e_1,\uu_1\}\subset g\algA$;

\item[(b)] $\{e_1,\vv_1\}\subset g\algA$;

\item[(c)] $\{e_1,e_2\}\subset g\algA$.
\end{enumerate}  
\end{lemma}
\begin{proof}
By Lemma~\ref{lemma1_algO_2} we can assume that $e_1\in\algA$.  There exists $b\in\algA$ such that $\{e_1,b\}$ are linearly independent.  Moreover, one can also assume that $b=\matr{0}{\uu}{\vv}{\be}$ for some $\uu,\vv\in\FF^3$ and $\be\in\FF$. 

Assume $\uu\neq\zero$. Since $e_1 b = \matr{0}{\uu}{\zero}{0}$, by part (a) of Remark~\ref{remark_Can1.2} there exists $g\in\SL_3$ such that $g(e_1,e_1 b)=(e_1,\uu_1)$, i.e., the case (a) holds. 

Assume $\vv\neq\zero$. Since $b\, e_1 = \matr{0}{\zero}{\vv}{0}$, by part (b) of Remark~\ref{remark_Can1.2} there exists $g\in\SL_3$ such that $g( e_1,b\,e_1)=(e_1,\vv_1)$, i.e., the case (b) holds.

In case $\uu=\vv=\zero$ we have $\be\neq0$, i.e., the case (c) holds.
\end{proof}

\begin{lemma}\label{lemma3_algO_1}
If $\dim\algA\geq3$, then there exists $g\in\G$ such that one of the following possibilities holds:
\begin{enumerate}
\item[(a)] $\{e_1,e_2,\uu_1\}\subset g\algA$;

\item[(b)] $\{e_1,\uu_1,\vv_2\}\subset g\algA$.
\end{enumerate}  
\end{lemma}

\medskip{}
\noindent{}Before the proof of this lemma we formulate the following remark.

\begin{remark}\label{remark_to_lemma3_algO_1}
\begin{enumerate}
\item[(a)] $\{e_1,e_2,\uu_1\}\subset \algA$ if and only if $\{e_1,e_2,\vv_1\}\subset \hbar\, \algA$;

\item[(b)] $\{e_1,\uu_1,\vv_2\}\subset \algA$ if and only if $\{e_1,\uu_2,\vv_1\}\subset g\algA$ for some $g\in\G$ (see part (c) of Remark~\ref{remark_Can1.2}).
\end{enumerate}  
\end{remark}

\begin{proof_of}{of Lemma~\ref{lemma3_algO_1}.}  By Lemma~\ref{lemma2_algO_1}, we assume that one of the following possibilities holds:

\medskip
\noindent{\bf 1.} $\{e_1,\uu_1\}\subset \algA$. There exists $b\in\algA$ such that $\{e_1,\uu_1,b\}$ are linearly independent.  Moreover, we can assume that $b=\matr{0}{\uu}{\vv}{\be}$ for some $\uu=(0,\ast,\ast)\in\FF^3$, $\vv=(\ga_1,\ga_2,\ga_3)\in\FF^3$ and $\be\in\FF$. 

Assume $\uu\neq\zero$. Since $e_1 b = \matr{0}{\uu}{\zero}{0}$, by part (d) of Remark~\ref{remark_Can1.2} there exists $g\in\SL_3$ such that $g(e_1,\uu_1,e_1 b)=(e_1,\uu_1,\uu_2)$. \tb{By part (c) of Remark~\ref{remark_Can1.2} there exists $g'\in\SL_3$ such that $g'(e_1,\uu_1,e_1 b)=(e_1,\uu_1,\uu_3)$}. The equality $\uu_1\uu_3=-\vv_2$ implies that the case (b) holds. 

Assume $\uu=\zero$.  Note that $b\, e_1 = \matr{0}{\zero}{\vv}{0}$ lies in $\algA$. 

Let $\ga_1\neq0$. The equality $b\, \uu_1=\ga_1 e_2$ implies that the case (a) holds. 

Otherwise, $\ga_1=0$. In case $\ga_2\neq0$ or $\ga_3\neq0$, by \tb{part (e) of Remark~\ref{remark_Can1.2}} there exists $g\in\SL_3$ such that $g(e_1,\uu_1, b\,e_1)=(e_1,\uu_1,\vv_2)$, i.e., the case (b) holds. 
If $\ga_2=\ga_3=0$, then $\be\neq0$ and $e_2\in\algA$, i.e., the case (a) holds.

\medskip
\noindent{\bf 2.} The case $\{e_1,\vv_1\}\subset \algA$ is similar to case 1.

\medskip
\noindent{\bf 3.} $\{e_1,e_2\}\subset \algA$.  There exists $b=\matr{\al}{\uu}{\vv}{\be}$ in $\algA$ such that $\{e_1,e_2,b\}$ are linearly independent. Moreover, we can assume that $\al=\be=0$.

Assume $\uu\neq\zero$. Since $e_1 b = \matr{0}{\uu}{\zero}{0}$, by part (a) of Remark~\ref{remark_Can1.2} there exists $g\in\SL_3$ such that $g(e_1,e_2,e_1 b)=(e_1,e_2,\uu_1)$, i.e., the case (a) holds. 

Otherwise, $\vv\neq\zero$.  Since $b = \matr{0}{\zero}{\vv}{0}$  by part (b) of Remark~\ref{remark_Can1.2} there exists $g\in\SL_3$ such that $g(e_1,e_2,b )=(e_1,e_2,\vv_1)$. By part (a) of Remark~\ref{remark_to_lemma3_algO_1} we obtain that case (a) holds. 
\end{proof_of}

\subsection{Proof of Proposition~\ref{prop_subalg}}\label{section_proof}
Assume $\mathfrak{A}=[\algA]$ for some subalgebra $\algA$ of $\OO$. Lemmas~\ref{lemma1_algO_2}, \ref{lemma2_algO_2}, \ref{lemma3_algO_2}, \ref{lemma2_algO_1}, \ref{lemma3_algO_1} imply that there exist $g\in\G$ such that $g\algA$ contains one of the sets from the formulation of Proposition~\ref{prop_subalg}. Since the  $\FF$-span of each of these sets is an algebra, the proof is completed.

\section{Separating invariants in case $\Char{\FF}=2$}\label{section_separ2}

In this section we assume that $\Char{\FF}=2$. Introduce the following notations for $\un{a}\in\OO^n$:
\begin{enumerate}
\item[$\bullet$]  the {\it rank} $\rk(\un{a})$ is the dimension of the subspace of $\OO$ spanned by $a_1,\ldots,a_n$; 

\item[$\bullet$]  $\alg(\un{a})$ is the $\FF$-algebra (in general, non-unital) generated by $a_1,\ldots,a_n$.

\end{enumerate}
Obviously, $\rk(g\un{a})=\rk(\un{a})$ for every $g\in\G$. The following remark is well-known (for example, see Theorem 2.3.6 of~\cite{derksen2002computationalv2}).

\begin{remark}\label{remark_closure}
Assume $\un{a}\in\OO^n$. Then there exists a unique  closed $\G$-orbit $O=O_{\un{a}}$ in the closure of $\G\un{a}$. Moreover, $O_{\un{a}}$ is the only closed orbit in the fiber
$$\{\un{c}\in\OO^n\,|\, f(\un{a})=f(\un{c}) \text{ for all } f\in K_n^{\G}\}.$$ 
In particular, $f(\un{a})=f(\un{c})$ for every $f\in K_n^{\G}$ and $\un{c}\in O_a$. 
\end{remark}

Observe that the group $\GL_n$ acts naturally on $\OO^n$ on the right as follows: for any $A=(\al_{ij})\in \GL_n$ and $\un{a}\in\OO^n$ we set
$$(\un{a}A)_i=\sum_{1\leq k\leq n} \al_{ki}a_k \;\;\text{ for }\;\; 1\leq i\leq n. $$
This action commutes with the left $\G$-action.

\begin{lemma}\label{lemma_separ_lin}
Given $\un{a},\un{b}\in\OO^n$, define $\un{a}'=\un{a}A$ and $\un{b}'=\un{b}A$ for some $A\in \GL_n$. Then
\begin{enumerate}
\item[(a)] $\G\un{a}=\G\un{b}$ if and only if $\G\un{a}'=\G\un{b}'$;

\item[(b)] given some $d\geq2$, we have that $\un{a}$ and $\un{b}$ are not separated by $\set{n}^{(d)}$ if and only if $\un{a}'$ and $\un{b}'$ are not separated by $\set{n}^{(d)}$;

\item[(c)] $\G\un{a}$ is closed if and only if $\G\un{a}'$ is closed. 
\end{enumerate}	
\end{lemma}
\begin{proof} Since $A$ is invertible, for each part of this lemma it is sufficient to prove ``only if''{} implication.

\medskip
\noindent{\bf (a)} For each $g\in \G$ the equality $g\un{a}=\un{b}$ implies $g\un{a}'=\un{b}'$, hence our claim follows. 

\medskip
\noindent{\bf (b)} Assume that $\un{a}$ and $\un{b}$ are not separated by $\set{n}^{(d)}$, i.e.,  $f(\un{a})=f(\un{b})$ for all $f\in \set{n}^{(d)}$.
The linearity of the trace together with \tb{Lemma~\ref{lemma1} and formulas}~\Ref{eq3a}, \tb{\Ref{eq_n}} imply that $h(\un{a}')=h(\un{b}')$ for all $h\in\set{n}^{(d)}$. 

\medskip
\noindent{\bf (c)} The right action by $A$ on $\OO^n$ gives a homeomorphism
of $\OO^n$ with respect to the Zariski topology. Hence it sends closed subsets to closed subsets.  \tb{Moreover, it maps $\G$-orbits to $\G$-orbits.}
\end{proof}


The following remark is a consequence of part (c) of Lemma~\ref{lemma_separ_lin}.

\begin{remark}\label{remak_closed}
An equivalence class $\mathfrak{A}\in\Omega(\OO)$ is closed if and only if for every subalgebra $\algA$ of $\OO$ with $[\algA]=\mathfrak{A}$ we have that if $\algA$ is the $\FF$-span of some $a_1,\ldots,a_n$, then the $\G$-orbit of $(a_1,\ldots,a_n)$ is closed in $\OO^n$. 
\end{remark}

\begin{prop}\label{prop_key} The set $\set{m}^{(8)}\subset \K{m}^{\G}$ is separating for every $m>0$ if and only if $\set{n}^{(8)}$ separates different closed $\G$-orbits of $\un{a}=(a_1,\ldots,a_l,0,\ldots,0)\in\OO^n$ and $\un{b}\in\OO^n$ for all $n>0$, where
\begin{enumerate}
\item[$\bullet$] $a_1,\ldots,a_l$ is a basis of some subalgebra $\algA$ of $\OO$,

\item[$\bullet$] $b_1,\ldots,b_n$ of $\OO$ linearly generate some subalgebra $\algB$ of $\OO$,

\item[$\bullet$] $\dim\algA \geq \dim\algB$.

\end{enumerate}
\end{prop}


\begin{proof}
We only have to prove ``if''{} part of the statement. Assume that $\un{a},\un{b}\in\OO^n$ are not separated by $\set{n}^{(8)}$ for some $n>0$. To obtain the required, we will show that $\G\un{a}=\G\un{b}$. 

By Remark~\ref{remark_closure}  we can assume that $\G\un{a}$ and $\G\un{b}$ are closed. 

\medskip
{\noindent \it Claim 1.} Given an $\FF$-basis $a'_1,\ldots,a'_l$ of $\FF$-span of $a_1,\ldots,a_n$, without loss of generality, we can assume that $\un{a}=(a'_1,\ldots,a'_l,0,\ldots,0)\in\OO^n$. 
\medskip

To prove Claim 1, we consider $A\in\GL_n$ such that $\un{a}A=(a'_1,\ldots,a'_l,0,\ldots,0)$. Parts (a), (b), (c) of Lemma~\ref{lemma_separ_lin} imply that we can consider $\un{a}A,\un{b}A$ instead of $\un{a},\un{b}$ and Claim 1 is proven.
\medskip

Denote by $\algA$ the algebra generated by $a_1,\ldots,a_n$ and by $\algB$ the algebra generated by $b_1,\ldots,b_n$. Without loss of generality we can assume that $\dim{\algA}\geq \dim{\algB}$. 

\medskip
{\noindent \it Claim 2.} Without loss of generality, we can assume that $\FF$-span of $a_1,\ldots,a_n$ is $\algA$ and $\FF$-span of $b_1,\ldots,b_n$ is $\algB$.
\medskip

Let us prove Claim 2. It is an easy exercise in the linear algebra that there exists $A\in\GL_n$ such that $\un{a}A=(a'_1,\ldots,a'_l,0,\ldots,0)$ and $\un{b}A=(0,\ldots,0,b'_d,\ldots,b'_t,0,\ldots,0)$, where $a'_1,\ldots,a'_l$ is a basis for $\FF$-span of $a_1,\ldots,a_n$ and $b'_d,\ldots,b'_t$ is a basis for $b_1,\ldots,b_n$.  Similarly to Claim 1, without loss of generality, we can take $\un{a}A,\un{b}A$ instead of $\un{a},\un{b}$, i.e., we assume that 
$$\un{a}=(a_1,\ldots,a_l,0,\ldots,0) \text{ and }\un{b}=(0,\ldots,0,b_d,\ldots,b_t,0,\ldots,0),$$ 
where $l\leq 8=\dim\OO$ and $t-d+1\leq 8$. There exist words $v_{1},\ldots,v_{r}$ of  $\FF\{\XX\}_n$ such that the $\FF$-span of the set $a_1,\ldots,a_n,v_{1}(\un{a}), \ldots, v_{r}(\un{a})$ is $\algA$. Similarly, there exist words $w_{1},\ldots,w_{s}$ of  $\FF\{\XX\}_n$ such that the $\FF$-span of the set $b_1,\ldots,b_n,w_{1}(\un{b}), \ldots, w_{s}(\un{b})$ is $\algB$. 

\tb{Since the map $\OO^n\to \OO^{r+s}$ given by $\un{x}\to (v_1(\un{x}), \ldots, v_{r}(\un{x}), w_{1}(\un{x}), \ldots, w_{s}(\un{x}))$ is a morphism of affine algebraic varieties}, the $\G$-orbits of  
$$\un{c}_1=(a_1,\ldots,a_n,v_{1}(\un{a}), \ldots, v_{r}(\un{a}), w_{1}(\un{a}), \ldots, w_{s}(\un{a})),$$  
$$\un{c}_2=(b_1,\ldots,b_n,v_{1}(\un{b}), \ldots, v_{r}(\un{b}), w_{1}(\un{b}), \ldots, w_{s}(\un{b}))$$ 
are closed. Obviously, $\G\un{a}=\G\un{b}$ if and only if $\G \un{c}_1=\G \un{c}_2$. By Lemma~\ref{lemma1} and formula~\Ref{eq_n}, for any $f\in \set{n+r+s}^{(8)}$ we have that $f(\un{c}_1)$ is a non-associative polynomial in $\tr((\ldots(a_{i_1} a_{i_2})\ldots )a_{i_k})$ and $n(a_i)$ for $1\leq i_1<\cdots<i_k\leq n$ and $1\leq i\leq n$. But this trace is zero in case $k>8$ by the construction of $\un{a}$. The same fact holds also for $f(\un{c}_2)$. Thus,   
$\un{a}$ and $\un{b}$ are not separated by $\set{n}^{(8)}$ if and only if $\un{c}_1$ and $\un{c}_2$ are not separated by $\set{n+r+s}^{(8)}$. Therefore, we can consider $\un{c}_1,\un{c}_2$ instead of $\un{a},\un{b}$ and Claim 2 is proven.

Since Claims 1 and 2 imply that $\un{a}$, $\un{b}$ satisfy conditions from the formulation of the lemma, we obtain that $\G\un{a}=\G\un{b}$. 
\end{proof}

\begin{lemma}\label{lemma_remark_1.3.1}
\begin{enumerate}
\item[1.] For every $a\in\MM$ with $\tr(a)=1$ and $n(a)=0$ there exists $g$ from the stabilizer $\Stab_{\G}(\MM)=\{g\in \G\,|\, g\,\MM\subset \MM\}$ such that $ga=e_1$.

\item[2.] For every $a\in\MM$ with $\tr(a)=0$ and $n(a)=1$ there exists $g\in\Stab_{\G}(\MM)$ such that $g a\in\{1_{\OO}, 1_{\OO} + \uu_1\}$.

\item[3.] Given non-zero $\ga\in\FF$, there exists $\xi_{\ga}\in  \Stab_{\G}(\MM)$ such that for every $\al_1,\ldots,\al_4\in\FF$ we have
$$\xi_{\ga}\matr{\al_1}{(\al_2,0,0)}{(\al_3,0,0)}{\al_4} = \matr{\al_1}{(\ga\al_2,0,0)}{(\ga^{-1} \al_3,0,0)}{\al_4}.$$

\item[4.] Assume that $\un{a}=(e_1,e_2)$ and $\un{b}\in\MM^2$ satisfy $\set{n}^{(2)}\!(\un{a})=\set{n}^{(2)}\!(\un{b})$. Then  there exists  $g\in \Stab_{\G}(\MM)$ such that $gb_1 = e_1$ and $gb_2 \in\{e_2, e_2 + \uu_1, e_2 + \vv_1\}$. 

\item[5.] If $b\in\MM$ satisfies $\tr(b)=n(b)=\tr(e_1 b)=0$, then $b\in \FF \uu_1$ or $b\in \FF \vv_1$.
\end{enumerate}
\end{lemma}
\begin{proof}
\medskip
\noindent{\bf 1.}  For $A=\mathcal{F}^{-1}(a)$ we have $\tr(A)=1$ and $\det(A)=0$. Hence there exists $g\in \GL_2$ such that $g^{-1} A g=\matr{1}{0}{0}{0}$ and Lemma~\ref{lemma_Grishkov} completes the proof. 

\medskip
\noindent{\bf 2.}  For $A=\mathcal{F}^{-1}(a)$ we have $\tr(A)=0$ and $\det(A)=1$. Hence there exists $g\in \GL_2$ such that $g^{-1} A g=\matr{\la}{1}{0}{\la}$ or $g^{-1} A g=\matr{\la_1}{0}{0}{\la_2}$ for some $\la,\la_1,\la_2$ and Lemma~\ref{lemma_Grishkov} completes the proof. 

\medskip
\noindent{\bf 3.} Given $g=\matr{1}{0}{0}{\ga}\in \GL_2$, we have 
$$g^{-1}\matr{\al_1}{\al_2}{\al_3}{\al_4} g = \matr{\al_1}{\ga\al_2}{\ga^{-1} \al_3}{\al_4}.$$
Lemma~\ref{lemma_Grishkov} concludes the proof.

\medskip
\noindent{\bf 4.} By part 1 we assume that $b_1=e_1$. Denote  $\mathcal{F}^{-1}(b_2)=B_2=\matr{\be_1}{\be_2}{\be_3}{\be_4}$. Since $0=\tr(a_1 a_2)=\tr(b_1 b_2)$, we obtain $\be_1=0$. The equalities $\tr(b_2)=1$ and $n(b_2)=0$ imply $\be_4=1$ and $\be_2\be_3=0$. Part 3 concludes the proof. 

\medskip
\noindent{\bf 5.} Denote  $\mathcal{F}^{-1}(b)=B=\matr{\be_1}{\be_2}{\be_3}{\be_4}$. Since $0=\tr(e_1 b)=\be_1$, the equalities $\tr(b)=n(b)=0$ conclude the proof.

\end{proof}

\begin{lemma}\label{lemma_separ1} 
Assume that $\un{a}=(a_1,0,\ldots,0)\in \MM^n$ and $\un{b}\in\MM^n$ are not  separated by $\set{n}^{(2)}$, where $a_1\in\{1_{\OO}, e_1\}$ and $\dim(\alg(\un{b}))\leq 1$. Then $\G\,\un{a} = \G\, \un{b}$.
\end{lemma}
\begin{proof} 

\noindent{\bf 1.} Let $a_1=1_{\OO}$. Since $\tr(b_1)=0$ and $n(b_1)=1$, by part 2 of Lemma~\ref{lemma_remark_1.3.1} we can assume that $b_1=1_{\OO}$ or $b_1=1_{\OO}+\uu_1$. 

In the first case the condition $\dim(\alg(\un{b}))\leq 1$ implies  $\un{b}=(1_{\OO}, \be_2 1_{\OO},\ldots, \be_n 1_{\OO})$ for some $\be_2,\ldots,\be_n\in\FF$. Since $0=n(b_i)=\be_i^2$ for all $1<i\leq n$, we have $\un{a}=\un{b}$.

In the second case we have that $b_1$ and $b_1^2=1_{\OO}$ are linearly independent; a contradiction.

\medskip
\noindent{\bf 2.} Let $a_1=e_1$. Since $\tr(b_1)=1$ and $n(b_1)=0$, by part 1  of Lemma~\ref{lemma_remark_1.3.1} we can assume that $b_1=e_1$. Then the condition $\dim(\alg(\un{b}))\leq 1$ implies that $\un{b}=(e_1, \be_2 e_1,\ldots, \be_n e_1)$ for some $\be_2,\ldots,\be_n\in\FF$. For each $1<i\leq n$ we have $0=\tr(a_1 0 )=\tr(b_1 b_i) = \be_i$. Therefore,  $\un{a}=\un{b}$.  
\end{proof}

\begin{lemma}\label{lemma_separ2} 
Assume that $\un{a}=(e_1,e_2,0,\ldots,0)\in \MM^n$ and $\un{b}\in\MM^n$ are not separated by $\set{n}^{(2)}$ and  $\dim(\alg(\un{b}))\leq 2$. Then $\G\,\un{a} = \G\, \un{b}$.
\end{lemma}
\begin{proof}  By part 4  of Lemma~\ref{lemma_remark_1.3.1} we can assume that $b_1=e_1$ and $b_2\in\{e_2, e_2 + \uu_1, e_2 + \vv_1\}$.

Let $b_2=e_2$. For $3\leq i\leq n$ part 5 of Lemma~\ref{lemma_remark_1.3.1} implies that $b_i\in \FF \uu_1$ or $b_i\in \FF \vv_1$, since $\tr(b_i)=n(b_i)=\tr(b_1 b_i)=0$. It follows from  $\dim(\alg(\un{b}))\leq 2$ that $b_i=0$ for all $3\leq i\leq n$. Therefore, $\un{a}=\un{b}$.

In case $b_2=e_2+\uu_1$ we consider $b_1b_2=\uu_1$ and obtain that $\{e_1,\uu_1,e_2\}\subset\alg(\un{b})$; a contradiction. 

In case $b_2=e_2+\vv_1$ we consider $b_2b_1=\vv_1$ and obtain that $\{e_1,\vv_1,e_2\}\subset\alg(\un{b})$; a contradiction. 
\end{proof}

\begin{lemma}\label{lemma_separ3} 
If $\un{a}=(e_1,e_2,\uu_1,\vv_1,0,\ldots,0)\in \MM^n$ and $\un{b}\in\MM^n$ are not separated by  $\set{n}^{(3)}$, then $\G\,\un{a} = \G\, \un{b}$. 
\end{lemma}
\begin{proof} By part 4  of Lemma~\ref{lemma_remark_1.3.1} we can assume that $b_1=e_1$ and $b_2\in\{e_2, e_2 + \uu_1, e_2 + \vv_1\}$. Assume $3\leq i\leq n$. We have  $\tr(b_i)=n(b_i)=\tr(b_1 b_i)=0$, since $\tr(a_1 a_3)=\tr(a_1 a_4)=0$. Thus part 5 of Lemma~\ref{lemma_remark_1.3.1} implies that $b_i= \be_i\uu_1$ or $b_i=\be_i \vv_1$ for some $\be_i\in\FF$. Since $\tr(b_3b_4)=\tr(a_3a_4)=1$, we obtain that $\be_3\be_4=1$ and either $b_3=\be_3 \uu_1$, $b_4=\be_4 \vv_1$ or $b_3=\be_3 \vv_1$, $b_4=\be_4 \uu_1$ for some non-zero \tb{$\be_3,\be_4\in\FF$ with $\be_3\be_4=1$}. \tb{Hence equalities $\tr(b_3 b_2) = \tr(b_4b_2) =0$ imply that $b_2 = e_2$.} 

\medskip
\noindent{\bf 1.} Let $b_3=\be_3 \uu_1$, $b_4=\be_3^{-1} \vv_1$. By part 3 of Lemma~\ref{lemma_remark_1.3.1} we can assume that $\be_3=1$. 

Consider $5\leq i\leq n$. If $b_i=\be_i\uu_1$, then $0=\tr(a_4 a_i)=\tr(b_4 b_i)=\be_i$.  If $b_i=\be_i\vv_1$, then $0=\tr(a_3 a_i)=\tr(b_3 b_i)=\be_i$. Therefore, $\un{a}=\un{b}$. 

\medskip
\noindent{\bf 2.} In case  $b_3=\be_3 \vv_1$, $b_4=\be_3^{-1} \uu_1$ we have $0=\tr((b_1b_3)b_4) = \tr((a_1 a_3) a_4)= \tr(\uu_1 \vv_1)=1$; a contradiction.
\end{proof}

\begin{lemma}\label{lemma_separ4} 
If $\un{a}=(e_1,e_2,\uu_1,\vv_1,\uu_2,\vv_2,\uu_3,\vv_3,0,\ldots,0)\in \OO^n$ and $\un{b}\in\OO^n$ are not separated by  $\set{n}^{(3)}$, then $\G\,\un{a} = \G\, \un{b}$. 
\end{lemma}
\begin{proof} Given $c_1,\ldots, c_8$, denote by $M_{c_1,\ldots,c_8}$ the Gram matrix $(\tr(c_ic_j))_{1\leq i,j\leq 8}$. Since the trace is a bilinear non-degenerate form on $\OO$ and $a_1,\ldots,a_8$ are linearly independent, we obtain that $\det(G_{a_1,\ldots,a_8})=\det(G_{b_1,\ldots,\tb{b_8}})$ is non-zero. Hence, $b_1,\ldots,b_8$ are also linearly independent. In particular,  $\FF$-span of $b_1,\ldots,b_8$ is $\OO$.

For every $1\leq i\leq 8$ and $8< j\leq n$ we have that $\tr(a_i a_j)=\tr(b_i b_j)$ is zero. Therefore, $\tr(bb_j)=0$ for all $b\in\OO$. Since $\tr$ is non-degenerate on $\OO$, we obtain $\un{b}=(b_1,\ldots,b_8,0,\ldots,0)$. 

For every $1\leq i,j\leq 8$ there exists $1\leq k_{ij}\leq 8$ and $\eta_{ij}\in\FF$ such that $a_i a_j=\eta_{ij} a_{k_{ij}}$. Therefore, for each $1\leq l\leq 8$ we have that $\tr((a_i a_j - \eta_{ij} a_{k_{ij}})a_l)=\tr((b_i b_j - \eta_{ij} b_{k_{ij}})b_l)$ is zero. Hence, $b_i b_j = \eta_{ij} b_{k_{ij}}$. Consider a linear map $f:\OO\to \OO$ defined on the basis of $\OO$ by $f(a_i)=b_i$ for all $1\leq i\leq 8$. Since the multiplication table for $a_1,\ldots,a_8$ is the same as for $b_1,\ldots,b_8$, we can see that $f\in\G$. The required is proven.
\end{proof}

The following statement is a corollary of Proposition~\ref{prop_subalg}.

\begin{cor}\label{cor_closed}
Assume $\Char{\FF}=2$ and a closed equivalence class $\mathfrak{A}\in\Omega(\OO)$ has the dimension $d\leq3$. Then one of the following sets is a basis for $\mathfrak{A}$: 
\begin{enumerate}
\item[$d=1${\rm:}] $\{1_{\OO}\}$, $\{e_1\}$;

\item[$d=2${\rm:}] $\{e_1,e_2\}$.
\end{enumerate}
\end{cor}
\begin{proof} 
We need to show that any basis $\{a_1,\ldots,a_n\}$ from Proposition~\ref{prop_subalg}, different from the above bases, generates non-closed equivalence class. For each $\un{a}=(a_1,\ldots,a_n)$ the arguments are the same: we find an element $\un{a}'$ in the closure of $\G\un{a}$ such that $\rk(\un{a}') < \rk(\un{a})$, that obviously implies that $\G\un{a}$ is not closed. 
\begin{enumerate}
\item[$\bullet$] If $\un{a}=(\uu_1)\in\OO^1$, then for the standard one-parameter subgroup $\theta_{\un{\la}}$ with $\un{\la}=(1,-1,0)$ the element $\un{a}'=\lim_{t\to 0} \theta_{\un{\la}}(t)\un{a}=(0)$ lies in $\overline{\G\un{a}}$ (see Section~\ref{subsection_1gr} for more details).

\item[$\bullet$] If $\un{a}=(1_{\OO},\uu_1)$, then  $\un{a}'=\lim_{t\to 0} \theta_{(1,-1,0)}(t)\un{a}=(1_{\OO},0)$ lies in $\overline{\G\un{a}}$. 

\item[$\bullet$] If $\un{a}=(\uu_1,\vv_2)$, then  $\un{a}'=\lim_{t\to 0} \theta_{(1,-1,0)}(t)\un{a}=(0,0)$ lies in $\overline{\G\un{a}}$. 

\item[$\bullet$] If $\un{a}=(e_1,\uu_1)$, then  $\un{a}'=\lim_{t\to 0} \theta_{(1,-1,0)}(t)\un{a}=(e_1,0)$ lies in $\overline{\G\un{a}}$. 

\item[$\bullet$] If $\un{a}=(e_1,\vv_1)$, then  $\un{a}'=\lim_{t\to 0} \theta_{(-1,1,0)}(t)\un{a}=(e_1,0)$ lies in $\overline{\G\un{a}}$.

\item[$\bullet$] If $\un{a}=(1_{\OO},\uu_1,\vv_2)$, then  $\un{a}'=\lim_{t\to 0} \theta_{(1,-1,0)}(t)\un{a}=(1_{\OO},0,0)$ lies in $\overline{\G\un{a}}$. 

\item[$\bullet$] If $\un{a}=(e_1,e_2,\uu_1)$, then  $\un{a}'=\lim_{t\to 0} \theta_{(1,-1,0)}(t)\un{a}=(e_1,e_2,0)$ lies in $\overline{\G\un{a}}$. 

\item[$\bullet$] If $\un{a}=(e_1,\uu_1,\vv_2)$, then  $\un{a}'=\lim_{t\to 0} \theta_{(1,-1,0)}(t)\un{a}=(e_1,0,0)$ lies in $\overline{\G\un{a}}$. 

\item[$\bullet$] If $\un{a}=(\uu_1,\vv_2,\vv_3)$, then  $\un{a}'=\lim_{t\to 0} \theta_{(1,-1,0)}(t)\un{a}=(0,0,\vv_3)$ lies in  $\overline{\G\un{a}}$. 
\end{enumerate}
\end{proof}

\begin{theo}\label{theo_separ2} 
The set $\set{n}^{(8)}$ is a separating set for $\K{n}^{\G}$ in case $\Char{\FF}=2$. 
\end{theo}
\begin{proof} We will apply Proposition~\ref{prop_key} to obtain the required statement. Assume that $\G$-orbits of   $\un{a}=(a_1,\ldots,a_l,0,\ldots,0)\in\OO^n$,  $\un{b}\in\OO^n$ are  closed,  $a_1,\ldots,a_l$ is a basis of some subalgebra $\algA$ of $\OO$, octonions $b_1,\ldots,b_n$ linearly generate some subalgebra $\algB$ of $\OO$, and  $\dim\algA \geq \dim\algB$. Moreover,  we assume that   $\un{a}$ and  $\un{b}$ are not separated by $\set{n}^{(8)}$. 
 
Let $\dim\algA=8$. We may choose that  $\un{a}=(e_1,e_2,\uu_1,\vv_1,\uu_2,\vv_2,\uu_3,\vv_3,0,\ldots,0)$ by Lemma~\ref{lemma_separ_lin}.  Then Lemma~\ref{lemma_separ4} implies that $\G\,\un{a}=\G\, \un{b}$.

Let $\dim\algA<8$. Then $\algA$ lies in a maximal proper subalgebra of $\OO$. By Theorem 5 of~\cite{Racine_1974}, the algebra of sextonions  
$\SE$ is the unique maximal proper subalgebra of $\OO$ modulo $\G$-action (see also Remark~\ref{remark_Racine} below). So $\algA\subset \SE$, i.e., for all $1\leq i\leq l$ we have 
$$a_i=\matr{\al_{i1}}{(\al_{i2}, \al_{i3}, 0)}{(\al_{i4}, 0, \al_{i5})}{\al_{i6}}
$$
for some $\al_{ij}\in\FF$. 
Similarly, we can assume that $\algB\subset \SE$. 

Since for the standard one-parameter subgroup $\theta_{\un{\la}}$ with $\un{\la}=(0,1,-1)$ we have 
$$\theta_{\un{\la}}(t)a_i = \matr{\al_{i1}}{(\al_{i2}, t \al_{i3}, 0)}{(\al_{i4}, 0, t \al_{i5})}{\al_{i6}},
$$%
then the limit  $\un{a}'=\lim_{t\to 0} \theta_{\un{\la}}(t)\un{a}$ exists (see Section~\ref{subsection_1gr} for more details). Obviously, $\un{a}' =(a'_1,\ldots,a'_l,0,\ldots,0)$ lies in $\MM^n$. The orbit $\G\un{a}$ is closed, therefore, $\un{a}'\in\G\,\un{a}$. Replacing $\un{a}$ by $\un{a}'$ we may assume that $\un{a}\subset \MM^n$. Therefore, $\algA\subset \MM$. In the same manner we can assume that $\algB\subset \MM$.

In case $\dim\algA=4$ by Lemma~\ref{lemma_separ_lin} we may choose that $\un{a}=(e_1,e_2,\uu_1,\vv_1,0,\ldots,0)$  and Lemma~\ref{lemma_separ3} implies that $\G\,\un{a} = \G\, \un{b}$. 

Let $\dim\algA\leq3$. By Corollary~\ref{cor_closed} and Lemma~\ref{lemma_separ_lin} we can assume that $\un{a}$ is one of the next elements: $(1_{\OO})$, $(e_1)$, $(e_1,e_2)$. If $\un{a}=(1_{\OO})$ or $\un{a}=(e_1)$, then Lemma~\ref{lemma_separ1} implies that  $\G\,\un{a} = \G\, \un{b}$.  If $\un{a}=(e_1,e_2)$, then Lemma~\ref{lemma_separ2} implies that  $\G\,\un{a} = \G\, \un{b}$. 

Finally, by Proposition~\ref{prop_key} the set $\set{n}^{(8)}$ is separating for $\K{n}^{\G}$.
\end{proof}

\begin{remark}\label{remark_Racine} 
\tb{In the proof of Theorem 5 of~\cite{Racine_1974}, which claims that $\SE$ is the unique maximal proper subalgebra of $\OO$ modulo $\G$-action, there is a small error, but this does not interfere with the case of an algebraically closed field. See~\cite{Gagola_2013},~\cite{Petersson_2013} for more details.}
\end{remark}

\begin{remark}\label{remark_separ2} 
It follows from Theorem~\ref{theo_separ2} that the set $\setO{n}^{(8)}$ is a separating set for $\KO{n}^{\G}$ in case $\Char{\FF}=2$. 
\end{remark}

\section*{Acknowledgement} 

We are very grateful to the anonymous referees for helpful comments.


\bigskip

\bibliographystyle{siam}
\bibliography{literature}

\end{document}